\documentclass[11pt]{amsart}
\usepackage{amscd, amsfonts, amssymb}
\usepackage{fullpage}
\usepackage{latexsym}
\usepackage{verbatim}
\usepackage{enumerate}
\usepackage{tikz}
\usetikzlibrary{arrows,matrix,decorations,positioning,backgrounds,fit}

\DeclareMathOperator{\cha}{char}

\DeclareMathOperator{\Ad}{Ad}
\DeclareMathOperator{\Int}{Int}
\DeclareMathOperator{\Cent}{Cent}
\DeclareMathOperator{\cent}{\mathfrak{c}}
\DeclareMathOperator{\Rad}{Rad}

\DeclareMathOperator{\Lie}{Lie}

\DeclareMathOperator{\coker}{coker}

\DeclareMathOperator{\Aut}{Aut}

\DeclareMathOperator{\Ext}{Ext}
\DeclareMathOperator{\Hom}{Hom}
\DeclareMathOperator{\Der}{Der}
\DeclareMathOperator{\SL}{SL}
\DeclareMathOperator{\PSL}{PSL}
\DeclareMathOperator{\GL}{GL}

\newcommand {\ZZ} {\mathbb{Z}}
\newcommand {\CC} {\mathbb{C}}
\newcommand {\QQ} {\mathbb{Q}}

\newcommand {\RR} {\mathbb{R}}

\newcommand {\Fp} {\mathbb{F}_p}
\newcommand {\Gm} {\mathbb{G}_m}

\newcommand {\N} {\mathcal{N}}
\newcommand {\U} {\mathcal{U}}

\newcommand {\m} {\mathfrak{m}}

\newcommand {\h} {\mathfrak{h}}
\newcommand {\g} {\mathfrak{g}}
\newcommand {\gl} {\mathfrak{gl}}

\newcommand {\p} {\mathfrak{p}}

\newcommand {\R} {\mathcal{R}} 

\newtheorem  {thm}[equation]        {Theorem}
\newtheorem* {thm*}                {Theorem}

\newtheorem  {corollary}[equation]   {Corollary}
\newtheorem* {corollary*}           {Corollary}
\newtheorem  {proposition}[equation] {Proposition}
\newtheorem  {lemma}[equation]       {Lemma}
\newtheorem*  {lemma*}       {Lemma}

\theoremstyle{definition}
\newtheorem  {definition}[equation]  {Definition}
\newtheorem* {definition*}          {Definition}

\newtheorem  {example}[equation]     {Example}
\newtheorem* {example*}             {Example}

\newtheorem* {examples*}             {Examples}

\newtheorem* {notation*}            {Notation}

\newtheorem* {exercise*}            {Exercise}

\theoremstyle{remark}
\newtheorem  {remark}[equation]      {Remark}
\newtheorem* {remark*}              {Remark}
\newtheorem  {remarks}[equation]      {Remarks}
\newtheorem* {remarks*}              {Remarks}

\numberwithin{equation}{section}


\title[On the smoothness of centralizers in reductive groups]
{On the smoothness of centralizers \\ in reductive groups}

\author[S. Herpel]{Sebastian Herpel}
\address
{Fakult\"at f\"ur Mathematik,
Ruhr-Universit\"at Bochum,
44780 Bochum, Germany}
\email{sebastian.herpel@rub.de}

\begin{document}

\begin{abstract}
Let $G$ be a connected reductive algebraic group over an algebraically closed
field $k$. In a recent paper, Bate, Martin, R\"ohrle and Tange show that every
(smooth) subgroup of $G$ is separable provided that the characteristic of $k$
is very good for $G$. Here separability of a subgroup means that its
scheme-theoretic centralizer in $G$ is smooth. Serre suggested extending this
result to arbitrary, possibly non-smooth, subgroup schemes of $G$. The aim of
this note is to prove this more general result. Moreover, we provide 
a condition on the characteristic of $k$ that is necessary and sufficient 
for the smoothness of all centralizers in $G$. We finally relate this condition
to other standard hypotheses on connected reductive groups.
\end{abstract}

\maketitle

\section{Introduction}
Let $G$ be a connected reductive algebraic group over an algebraically closed
field $k$. A closed subgroup $H \subseteq G$ is called \emph{separable in $G$}
if the Lie algebra of its centralizer coincides with the infinitesimal
fixed points of $H$ in $\Lie(G)$. Here, in contrast to the later
sections of the paper, we do not consider scheme-theoretic centralizers. A
similar notion exists for subalgebras of $\Lie(G)$.
See \cite{BMR:2005} and \cite{BMRT:2010} for these
concepts and their importance in the context of Serre's notion of $G$-complete
reducibility.  
The starting point of this note is the following result of Bate,
Martin, R\"ohrle and Tange (\cite[Thm.\ 1.2]{BMRT:2010}).
Liebeck-Seitz (\cite[Thm.\ 3]{LS:1996})
and Lawther-Testerman (\cite[Thm.\ 2]{LT:1999})
earlier obtained special cases of this result.

\begin{thm*}
Suppose that $\cha k$ is very good for $G$. Then any subgroup of $G$ is
separable in $G$ and any subalgebra of $\Lie(G)$ is separable in $\Lie(G)$.
\end{thm*}

Separability of a subgroup $H$ of $G$ can also be characterized as the
smoothness of the scheme-theoretic centralizer of $H$ in $G$, see Section \ref{sec:sep}. 
In particular,
all subgroups are separable if the characteristic of the ground field is zero
(all algebraic group schemes in characteristic zero are smooth, due to a
theorem of Cartier). The above theorem is therefore an instance of the general
principle that a characteristic zero result becomes true in positive
characteristic, provided that the characteristic is ``big enough''.  More
precisely, the assertion of the theorem then can be interpreted in the
following way: If $H$ is an arbitrary closed subgroup scheme of $G$ and if
$\cha k$ is very good for $G$, then the scheme-theoretic centralizer of $H$ in
$G$ is smooth provided that either $H$ is smooth or $H$ is infinitesimal of
height one (\cite[Rem.\ 3.5(vi)]{BMRT:2010}, see Lemma \ref{lem:sep=smooth}
below).

We are going to extend this result to arbitrary closed subgroup schemes, and
to the weaker 
condition that the characteristic is only 
``pretty good for $G$'' (Definition \ref{def:G-vg}); see Theorem \ref{thm:main}. 
We are also going to exhibit examples of non-smooth centralizers in small
characteristic. In \cite[Ex.\ 3.11]{BMRT:2010}, the authors construct
non-separable subgroups of simple groups of type $G_2$ and $F_4$ in
characteristic two. We observe that non-separable subgroups may be found in any
reductive group in non pretty good characteristic; see Example \ref{ex:counterexamples}.

Combining these results, the following is our main theorem:
\begin{thm}\label{thm:examples}
Let $G$ be a connected reductive algebraic group. Then the characteristic of $k$ is 
zero or pretty good for $G$ if and only if all centralizers of closed subgroup schemes in $G$
are smooth. 
\end{thm}

Finally, we show that the universal smoothness of centralizers holds in many classes 
of ``standard'' reductive groups,
as introduced by Jantzen and McNinch-Testerman; see Corollary \ref{cor:univsm}. 
In fact, the different notions of standardness almost coincide (Theorem \ref{thm:standard}). 
This demonstrates that requiring
the characteristic of the ground field to be pretty good for $G$ is a natural assumption
for connected reductive groups.

The paper is organized as follows. 
In Section \ref{sec:prelim}, we recall general facts about
group schemes. We also recall the notions of good and very good primes. 
In Section \ref{sec:vg},
we
introduce the
notion of a pretty good prime.

In Section \ref{sec:smoothcent}, we investigate conditions for the smoothness of centralizers. 
First, in Section \ref{sec:sep}, we relate the above notion of separability to the 
smoothness of certain scheme-theoretic centralizers. In Section \ref{sec:necsm}, it is
our goal to prove the forward implication of Theorem \ref{thm:examples}. In fact, we prove
a slightly more general result in Theorem \ref{thm:main}.

In Section \ref{sec:nonsmooth}, we provide methods to construct non-smooth centralizers outside
of pretty good characteristic. This allows us to give a proof of Theorem \ref{thm:examples} at the end
of this section.

Finally, in Section \ref{sec:standard}, we recall other notions of standard reductive groups, and prove
that these amount to essentially the same as requiring the group to be defined in pretty good
characteristic.

\section{Preliminaries} \label{sec:prelim}

In this section, let $k$ be a field and $\overline{k}$ a fixed algebraic
closure of $k$. Let $S$ be a commutative ring.

\subsection{Basic definitions} \label{subsec:basic}
We call a functor from the category of $S$-algebras to the category of sets
(resp.\ groups) an \emph{$S$-functor} (resp.\ \emph{$S$-group functor}). We
adopt the point of view that (group) schemes over $S$ are certain $S$-(group)
functors. See \cite{DG:1970} and \cite{Jantzen:2003} for this approach. In
particular, morphisms of (group) schemes are just morphisms of $S$-(group)
functors. If $X$ is an $S$-(group)-functor and if $r:R\rightarrow R'$ is a
morphism of $S$-algebras, we write $x_r$ for the image of an element $x \in
X(R)$ under the induced map $X(R) \rightarrow X(R')$. Since we are mainly
interested in affine objects, let us recall the basic definitions: An
\emph{affine $S$-(group) scheme} is an $S$-(group) functor $X$ of the form $X =
\Hom(A,\_)$, where $A$ is an $S$-(Hopf) algebra (and where $\Hom$ denotes
homomorphisms of $S$-algebras). We also write $S[X]$ instead of $A$. We call
$X$ \emph{of finite type} or \emph{algebraic} if $S[X]$ is a finitely generated
$S$-algebra. If $R$ is an $S$-algebra, every affine $S$-(group) scheme $X$
gives rise to an affine $R$-(group) scheme $X_R$, with $R[X_R] = S[X]
\otimes_{S} R$. If $M$ is an $S$-module, we define the $S$-group functor $M_a$
via $M_a(R) = M \otimes_S R$, see \cite[II, \S 1, 2.1]{DG:1970}. If $M$ is
projective and finite, then $M_a$ is an affine $S$-group scheme represented by
the symmetric algebra $S(M^*)$.

\subsection{Notions for group morphisms} \label{subsec:morphisms}

Let  $\varphi: G \rightarrow G'$ be a morphism of affine algebraic $k$-group
schemes. It is called \emph{surjective} if the comorphism $\varphi^*:k[G']
\rightarrow k[G]$ is injective (which usually does $not$ imply surjectivity of
all induced maps $G(R) \rightarrow G'(R)$). Surjectivity can be characterized
by the following property (see \cite[15.5]{Water:1979}): For every $k$-algebra
$R$ and every $g' \in G'(R)$ there is a faithfully flat extension $s:R
\rightarrow \overline{R}$ and an element $g\in G(\overline{R})$ such that
$\varphi(\overline{R})(g)=g'_s$. 
  
Let $H \subseteq G$ and $H' \subseteq G'$ be closed subgroup schemes. The
functor $\varphi^{-1} H': R \mapsto \varphi(R)^{-1}(H'(R))$ is a closed
subgroup scheme of $G$, defined by the ideal generated by $\varphi^*(I')$, if
$I' \subseteq k[G']$ defines $H'$. The induced map $\varphi^{-1} H'
\rightarrow H'$ is surjective if $\varphi$ is surjective (see \cite[VI, 22.4]{KMRT:1998}). 
The \emph{kernel} of $\varphi$ is the closed subgroup scheme $\ker
\varphi := \varphi^{-1} \{e\}$. Let $\varphi H$ be the closed subgroup scheme
of $G'$ defined by the ideal $(\varphi^*)^{-1}(I)$, if $I \subseteq k[G]$
defines $H$. Note that $\varphi(R) (H(R)) \subseteq (\varphi H)(R)$ for all
$k$-algebras $R$. The induced map $H \rightarrow \varphi H$ then is surjective
in the above sense, since its comorphism is injective. In particular,
\begin{align*} (\varphi H)(R) = \{x \in G'(R) \mid \text{ there is a faithfully
flat } s:R\rightarrow \overline{R} \text{ such that }  x_s \in
\varphi(\overline{R})(H(\overline{R}))\}, \end{align*} 
where the forward containment follows from the above
characterization of surjectivity; the reverse inclusion follows from the 
observation that
$x_s (\varphi^*)^{-1}(I) = 0$ implies $x(\varphi^*)^{-1}(I) = 0$.

A sequence of morphisms of affine algebraic group schemes $1 \rightarrow N
\xrightarrow{\varphi} G \xrightarrow{\psi} H \rightarrow 1$ is called
\emph{exact} provided that $\psi$ is surjective and $\varphi$ induces an
isomorphism of $N$ with $\ker(\psi)$. 
A morphism $\pi: G \rightarrow G'$ of affine algebraic $k$-group schemes is
called an \emph{isogeny} if it is surjective and if its kernel is a finite
group scheme.

\subsection{Fixed points and centralizers} \label{subsec:centralizers}
If $G$ is a $k$-group functor and $X$ is a $k$-functor, there is an obvious
notion of a natural action $\alpha: G \times X \rightarrow X$. Suppose such an
action is given. Then we can define the \emph{fixed point functor} $X^G$. It
associates to each $k$-algebra $R$ the set \begin{align*} X^G(R) = \{ x \in
X(R) \mid \alpha(R')(g,x_r) = x_r \text{ for every } r:R \rightarrow R' \text{
and }g \in G(R')\}.  \end{align*} We also define the \emph{centralizer} of a
subfunctor $Y$ of $X$ in $G$ as the functor $\Cent_G(Y)$ that associates to
each $k$-algebra $R$ the group \begin{align*} \Cent_G(Y)(R) = \{ g \in G(R)
\mid \alpha(R')(g_r,y) = y \text{ for every } r:R \rightarrow R' \text{ and }y
\in Y(R')\}.  \end{align*} It is well known that if $G$, $X$ and $Y$ are
$k$-schemes and if $X$ is separated, then $X^G$ is closed in $X$, and
$\Cent_G(Y)$ is closed in $G$ (see \cite[I, 2.6]{Jantzen:2003}). 

We are interested in two special cases of the above situation. If $H$ is a
closed subgroup scheme of an affine algebraic $k$-group scheme $G$, it acts on
$G$ by conjugation. We denote the corresponding fixed point functor by
$C_G(H)$. It is the same as the centralizer $\Cent_G(H)$, where $G$ acts on
itself via conjugation. We call $C_G(H)$ the \emph{centralizer of $H$ in $G$}.
We have the equality \begin{align} \label{eq:cent-base} C_G(H)_{\overline{k}} =
C_{G_{\overline{k}}}(H_{\overline{k}}).  \end{align} The \emph{center} of $G$
is defined by $Z(G) = C_G(G)$. 
An isogeny $\pi: G \rightarrow G'$ of affine algebraic $k$-group schemes
is called \emph{central} if $\ker \pi \subseteq Z(G)$. 

If $M \subseteq N$ are two $k$-vector spaces,
we can consider $M_a$ as a subfunctor of $N_a$. If $G$ acts linearly on
$N_a$ (i.e.\ $R$-linearly on each $N \otimes R$),
then $N$ is called a \emph{$G$-module}.
In this case,
we write $\Cent_G(M)$
for the centralizer $\Cent_G(M_a)$.

\subsection{Tangent spaces and Lie algebras} \label{subsec:tangent}
Let $X$ be an affine $S$-scheme. Let $S[\epsilon] = S[T]/T^2$ with
$\epsilon=\overline{T}$ denote the ring of dual numbers over $S$. Let
$p:S[\epsilon] \rightarrow S$ be the map defined by $\epsilon \mapsto 0$. Let
$x \in X(S)$. Then the (Zariski) \emph{tangent space} of $X$ in $x$ is defined
as a set by \begin{align*} T_x X = \{y \in X(S[\epsilon]) \mid y_p = x \},
\end{align*}    see \cite[II, \S4, 3.3]{DG:1970}. Using $S[X]$, we can write
this as \begin{align*} T_x X = \{x +g\epsilon \in X(S[\epsilon]) \mid g:S[X]
\rightarrow S \text{ is an $S$-derivation with respect to } x:S[X] \rightarrow
S\}, \end{align*} and we can endow $T_x X$ with the $S$-module structure
induced by the bijection with the space of all such point derivations. If
$\varphi: X \rightarrow Y$ is a morphism of affine schemes that sends $x \in
X(S)$ to $y \in Y(S)$, then it induces an $S$-linear map $d\varphi_x :=
\varphi(S[\epsilon])|_{T_x X}  :T_x X \rightarrow T_y Y$.

If $G$ is an affine algebraic $S$-group scheme, we write $\Lie(G) = \g =
T_e(G)$, where $e \in G(S)$ is the identity element. Suppose that $\cha(S)=p$.
Then the $S$-module $\g$ is finite and has the structure of a $p$-Lie algebra
over $S$, see \cite[II, \S7, 3.4]{DG:1970}, and we call it the \emph{Lie
algebra} of $G$. This construction has the following converse (see
\emph{Ibid.}, II, \S7, 3.5): If $\mathfrak{l}$ is a finite projective
$S$-module with a $p$-Lie algebra structure over $S$, then there is an affine
algebraic $S$-group scheme $\mathbb{G}(\mathfrak{l})$ with Lie algebra
isomorphic to $\mathfrak{l}$ such that taking differentials yields a bijection
\begin{align} \label{eq:lie-adjoint} \Hom(\mathbb{G}(\mathfrak{l}),G) \cong
\Hom(\mathfrak{l}, \g), \end{align} where on the right hand side we consider
homomorphisms of $p$-Lie algebras over $S$. Moreover, $\mathbb{G}$ commutes
with changing the base ring. If $S=k$ is a field, $\mathbb{G}$ induces an
equivalence between the category of finite-dimensional $p$-Lie algebras and the
category of algebraic $k$-group schemes of height one; then  for sub-$p$-Lie
algebras $\mathfrak{l} \subseteq \Lie(G)$ (where $G$ is an affine algebraic
$k$-group scheme) the group $\mathbb{G}(\mathfrak{l})$ can be identified with a
closed subgroup scheme of $G$ (\emph{Ibid.}, II, \S7, 4). 

If $G$ is an affine algebraic $k$-group scheme and if $S$ is a $k$-algebra,
then $\Lie(G) \otimes S \cong \Lie(G_S)$ (see \emph{Ibid.}, II, \S4, 4.8). In
the same way we can prove part (a) of the following lemma. The other parts then
follow from the definitions and an argument similar to \emph{Ibid.}, II, \S4,
2.5. 

\begin{lemma} \label{lem:tangent-space}
Let $X$ be an affine algebraic $k$-scheme. Let $x \in X(k)$. Let $G$ be an
affine algebraic $k$-group scheme. 
\begin{itemize} 
\item[(a)] The $k$-functors
$R \mapsto T_{x_R} X_R$ and $(T_x X)_a$ are isomorphic. Here $x_R$ denotes the
image of $x$ in $X(R)$. 
\item[(b)] If $G$ acts naturally on $X$ and if $x \in
X^G(k)$, then $G$ acts linearly on $T_x X$ and there is an equality $(T_x X)^G
= T_x(X^G)$. 
\item[(c)] Let $\varphi: X \rightarrow Y$ be a $G$-equivariant
morphism of affine algebraic $k$-schemes. Suppose that $x \in X^G(k)$. Then the
induced linear map $d \varphi_x: T_x X \rightarrow T_{\varphi(x)} Y$ is also
$G$-equivariant.  
\end{itemize} 
\end{lemma}

If $G$ is an affine algebraic $k$-group scheme and if $g \in G(R)$, then
conjugation with $g$ induces an automorphism $\Int(g):G_R \rightarrow G_R$,
sending $h \in G(R')$ to $g_r h g_r^{-1}$ for each morphism $r:R \rightarrow R'$
of $k$-algebras.
This in turn induces an automorphism $\Ad(g)=d \Int(g):\Lie(G_R) \rightarrow
\Lie(G_R)$. Using the above isomorphism of $k$-functors $\Lie(G_?) \cong \g_a$,
we see that $G$ acts linearly on $\g_a$ via $\Ad$. Now suppose that $H$ is a
closed subgroup scheme of $G$, and that $i:H \rightarrow G$ is the inclusion.
With the notation of Section \ref{subsec:centralizers} we find that
\begin{align} C_G(H)(R) &=\{ g \in G(R) \mid \Int(g) = i_R:H_R \rightarrow G_R
\}, \label{eq:cent1} \\ \Cent_G(\h)(R) &=\{ g \in G(R) \mid \Ad(g) =
di_R:\Lie(H_R) \rightarrow \Lie(G_R) \}. \label{eq:cent2} \end{align}

Lemma \ref{lem:tangent-space} (b) also implies that
\begin{align} \label{eq:lie_cent}
\Lie(C_G(H)) = \g^H.
\end{align}

\subsection{The identity component} \label{subsec:con-comp}
Let $G$ be an affine algebraic $k$-group scheme. There is a finite group
scheme $\pi_0 G$ which is represented by the largest separable subalgebra of
$k[G]$, see \cite[Thm.\ 6.7]{Water:1979}. Let $G^\circ$ denote the kernel of
the corresponding morphism $G \rightarrow \pi_0 G$. Then $G^\circ$ is a
connected closed normal subgroup scheme of $G$, called the \emph{identity
component} of $G$. 

Now let $G$ be an affine algebraic $S$-group scheme. For a prime ideal $\p
\subseteq S$, let $\kappa(\p)$ denote the residue field of the local ring
$S_\p$. Then we can construct the identity component of the affine
$\kappa(\p)$-group scheme $G_{\kappa(\p)}$. For $x \in G(R)$ let $x_\p$ denote
the image of $x$ in $G_{\kappa(\p)}(R \otimes \kappa(\p))$ under the obvious
map. We define a subgroup functor $G^\circ$ of $G$ via \begin{align*}
G^\circ(R) = \{x \in G(R) \mid x_\p \in G_{\kappa(\p)}^\circ(R \otimes
\kappa(\p)) \text{ for all prime ideals } \p \subseteq S \}, \end{align*} and
we call $G^\circ$ the \emph{identity component} of $G$. This construction
commutes with base change. In particular, for an affine algebraic $k$-group
scheme $G$ and a $k$-algebra $R$ we have \begin{align} \label{eq:comp-base}
(G^\circ)_R = (G_R)^\circ.  \end{align} In fact, our definition coincides with
the one in \cite[VI B, 3.1]{SGA3}. If $\varphi: G \rightarrow G'$ is a morphism
of affine algebraic $S$-group schemes, then it induces a morphism of $S$-group
functors $G^\circ \rightarrow (G')^\circ$.

\subsection{Smoothness and algebraic groups} \label{subsec:smoothness}
Let $X$ be an affine algebraic $k$-scheme and consider a point $x \in X(k)$. 
Let $\dim_x X$ be the dimension of the local ring $k[X]_x$ obtained by
localizing $k[X]$ at the kernel of $x$.
We say that $x \in X(k)$ is a
\emph{regular point} if $\dim_x X = \dim_k T_x X$. This is equivalent to the
regularity of $k[X]_x$. 
Using the dimension formula for flat morphisms (see \cite[I,
\S3, 6.3]{DG:1970}), we get the following lemma.

\begin{lemma} \label{lem:surjective-diff}
Let $\varphi:X \rightarrow Y$ be a flat morphism of affine algebraic
$k$-schemes. Suppose that $x \in X(k)$ and $y = \varphi(x) \in Y(k)$ are
regular points. Suppose that $x$ is also regular in the fibre
$\varphi^{-1}(\{y\})$. Then the induced map $d\varphi_x: T_x X \rightarrow T_y
Y$ is surjective (with kernel $T_x(\varphi^{-1}(\{y\}))$).  \end{lemma}

An affine algebraic $k$-scheme $X$ is called \emph{smooth} provided that
$X_{\overline{k}}$ is regular, i.e.\ all prime ideals in $k[X] \otimes
\overline{k}$ induce regular local rings. For affine algebraic $k$-group
schemes, this property can be characterized in the following way (see
\cite[21.8, 21.9]{KMRT:1998}, and \cite[II, \S5, 2.1]{DG:1970}):

\begin{proposition} \label{prop:smoothness}
An affine algebraic $k$-group scheme $G$ is smooth if and only if any of the
following statements hold: \begin{itemize} \item[(a)] $G^\circ$ is smooth,
\item[(b)] the identity element $e \in G(k)$ is a regular point, \item[(c)]
$\dim_k \g = \dim G$ (or, equivalently, $\dim_k \g \leq \dim G$), \item[(d)]
$k[G] \otimes \overline{k}$ is reduced.  \end{itemize} \end{proposition}

By a theorem of Cartier (see \cite[II, \S6, 1.1]{DG:1970}), every affine
algebraic $k$-group scheme $G$ is smooth if the characteristic of $k$ is zero.
We reserve the term \emph{algebraic group} over $k$ for a smooth affine
algebraic $k$-group scheme. There is an equivalence of categories between the
category of algebraic groups over $k$ and the category of linear algebraic
groups over $\overline{k}$ that are equipped with a $k$-structure (see \cite{Borel:1991} and
\cite{Springer:1998} for this language). This equivalence can be realized as
the functor $G \mapsto G(\overline{k})$, if we regard $G(\overline{k})$ as an
affine variety with coordinate ring $k[G] \otimes \overline{k}$; it gives a
correspondence between the smooth closed algebraic subgroups of $G$ and the
closed subgroups of $G(\overline{k})$ that are defined over $k$. Moreover, the
Lie algebras $\g \otimes \overline{k}$ and the Lie algebra associated to
$G(\overline{k})$ coincide. 

If $1 \rightarrow N \rightarrow G \rightarrow H \rightarrow 1$ is an exact
sequence of affine algebraic $k$-group schemes,
we have that $\dim G =
\dim N + \dim H$ and that $G$ is smooth if $H$ and $N$ are smooth (see
\cite[22.9-22.11]{KMRT:1998}).

An isogeny $G \rightarrow G'$ is called \emph{separable} if its kernel is smooth. Two affine algebraic
$k$-group schemes are called \emph{separably isogenous} if there exists a
separable isogeny between them.

\subsection{Reductive groups} 
We say that an affine algebraic $k$-group scheme
$G$ is a \emph{reductive} algebraic group over $k$, if it is smooth and if
$G(\overline{k})$ corresponds to a reductive linear algebraic group defined
over $k$ via the equivalence given in Section \ref{subsec:smoothness}.
Equivalently, $G_{\overline{k}}$ has no non-trivial closed smooth connected
normal unipotent subgroup schemes.

If $G$ is a reductive algebraic group, we denote by $\Rad(G)$ its radical (i.e.\
the unique maximal torus of the center of $G$), and by $\Der(G)$ the
derived group of $G$ (see \cite[II, \S5, 4.8]{DG:1970}).

An affine algebraic $k$-group scheme $G$ is called \emph{linearly reductive}
provided that all $G$-modules are semisimple. This is equivalent to the
vanishing of $H^1(G,V)$ for all finite-dimensional $G$-modules $V$ (see
\cite[II, \S3, 3.7]{DG:1970}). 

Suppose $(G',G)$ is a pair of reductive algebraic groups such that $G \subseteq
G'$ is a closed subgroup. Then $(G',G)$ is called a \emph{reductive pair} (see
Richardson \cite{Rich:1967}) provided that $\Lie(G')$ decomposes as a
$G$-module (via the adjoint action restricted from $G'$) into a direct sum
$\Lie(G') = \Lie(G) \oplus \m$.

\subsection{Root data} A \emph{root datum} $\R$ is a quadruple 
$\R=(X,\Phi,Y,\Phi^\vee)$, where $X$ and $Y$ are free abelian groups of
finite rank that are in duality relative to a pairing $\langle,\rangle : 
X \times Y \rightarrow \ZZ$, and where $\Phi \subseteq X$ and $\Phi^\vee \subseteq Y$
are finite sets. The following additional axioms are required: 
(i) there is a bijection $\alpha \mapsto \alpha^\vee$ of $\Phi$ onto $\Phi^\vee$ such
that $\langle \alpha,\alpha^\vee \rangle =2$; (ii) for each $\alpha \in \Phi$, the map 
$x \mapsto x-\langle x,\alpha^\vee\rangle \alpha$ (resp.\ 
$y \mapsto y-\langle \alpha,y\rangle \alpha^\vee$) stabilizes $\Phi$ (resp.\ $\Phi^\vee$).  
For a subset $\Phi' \subseteq \Phi$, we write $\Phi'^\vee = \{\alpha^\vee \mid \alpha \in \Phi' \}$.

A root datum $\R$ is called \emph{reduced} if $\alpha/2 \notin \Phi$ for all $\alpha \in \Phi$.
This means that the \emph{root systems} $\Phi \subseteq \RR\Phi$ and $\Phi^\vee \subseteq 
\RR \Phi^\vee$ are reduced.
A root datum is called \emph{semisimple} provided that the ranks of $X$ and $\ZZ\Phi$ coincide.
Suppose $\tilde{\R}=(\tilde{X},\tilde{\Phi},\tilde{Y},\tilde{\Phi}^\vee)$ is another root datum, $f: X \rightarrow 
\tilde{X}$ is a $\ZZ$-linear map, and $f^\vee$ is the induced map $\tilde{Y} \rightarrow Y$.
Then $f$ is called an \emph{isogeny of root data}, and denoted $f:\R \rightarrow \tilde{\R}$,
if $f$ is injective with finite cokernel, $f$ maps $\Phi$ onto $\tilde{\Phi}$ and $f^\vee$ maps
$\tilde{\Phi}^\vee$ onto $\Phi^\vee$. See \cite[XXI]{SGA3} for more generalities on root data.

The significance of these notions is given by the following: Consider the category of pairs
$(G,T)$, where $G$ is a connected reductive group over $\overline{k}$ and $T$ is a maximal
torus in $G$, and where morphisms are central isogenies respecting the torus. Considering
characters, roots, cocharacters and coroots relative to $T$, one associates a root datum
$\R(G,T)$ to $(G,T)$. In this way one obtains a contravariant functor to the category
of reduced root data with isogenies, which is almost an equivalence of categories (the functor
is not faithful, but two central isogenies giving the same isogenies of root data differ only
by conjugation with a torus element; to fix this one may add more data involving the 
choice of positive roots, see \cite[XXIII, Thm.\ 4.1]{SGA3}). 
The reductive group $G$ is semisimple if and only if its root datum is semisimple. Separable
isogenies between connected reductive groups are automatically central and correspond to
isogenies $f$ of root data such that $\coker(f)$ has order invertible in $k$.
If $G$ has root datum $\R = (X,\Phi,Y,\Phi^\vee)$, then the \emph{dual root datum} 
$\R^\vee = (Y,\Phi^\vee,X,\Phi)$ with the obvious pairing and bijection gives rise to the
\emph{dual group} $G^\vee$.

\subsection{Good and very good primes}
Let $\Phi$ be a (reduced) root system with irreducible components
$\Phi_1,\dots, \Phi_t$. Let $\Delta = \Delta_1 \cup \dots \cup \Delta_t$ be a
base of $\Phi$, where each $\Delta_i$ is a base of $\Phi_i$. 
For each $i$, write the highest root $\tilde{\alpha}_i$ of
$\Phi_i$ as a linear combination of simple roots. If a prime number $p$ divides
any coefficient occurring among the expressions of the $\tilde{\alpha}_i$, it
is said to be \emph{bad} (for the root system $\Phi$). Otherwise it is called
\emph{good}. By inspecting the coefficients of the highest roots of irreducible
root systems, one can verify that a prime $p$ is bad if and only if it occurs
among the coefficients of the highest roots, if and only if it is smaller than
the largest coefficient. A prime number $p$ is called \emph{very good} for
$\Phi$, if $p$ is good and if $p$ does not divide $n+1$ for any irreducible
component of type $A_n$ in $\Phi$. 
Write $\Lambda = \{\lambda \in \RR\Phi \mid \langle \lambda, \alpha^\vee \rangle
\in \ZZ \text{ for all } \alpha \in \Phi \}$ for the weight lattice of $\Phi$.
By inspecting the orders of the fundamental
groups of the irreducible root systems one can verify that a prime $p$ is very
good if and only if it is good and does not divide the order of the fundamental
group of $\Phi$. The following more conceptual characterization of these
notions is an easy generalization of \cite[I, 4.1-4.5]{SpSt:1970}.

\begin{lemma} \label{lem:good}
Let $\Phi$ be a reduced root system with weight lattice $\Lambda$. Let $p$ be a
prime number. Then the following hold: \begin{itemize} \item[(a)] $p$ is good
for $\Phi$ if and only if the $p$-torsion of $\ZZ \Phi / \ZZ \Phi'$ vanishes
for all subsets $\Phi' \subseteq \Phi$, \item[(b)] $p$ is very good for $\Phi$
if and only if the $p$-torsion of $\Lambda / \ZZ \Phi'$ vanishes for all
subsets $\Phi' \subseteq \Phi$.  \end{itemize} \end{lemma}

Now let $G$ be a reductive algebraic group over $k$ and let $\Phi$ be the root
system associated to $G_{\overline{k}}$. Then we call a prime number $p$
\emph{(very) good for $G$} provided that $p$ is (very) good for the root system
$\Phi$. 

\subsection{Pretty good primes for root data}\label{sec:vg}
Let $\R=(X,\Phi,Y,\Phi^\vee)$ be a (reduced) root datum. 
Motivated by Lemma \ref{lem:good}, we introduce the following terminology:

\begin{definition} \label{def:G-vg}
A prime number $p$ is called \emph{pretty good} for the root datum 
$\R$ if the groups
$X/\ZZ\Phi'$ and $Y/\ZZ\Phi'^\vee$ have no $p$-torsion, for all
subsets $\Phi'\subseteq \Phi$.
If $G$ is a connected reductive algebraic group over $k$, 
we say that a prime $p$ is \emph{pretty good for $G$} provided 
that it is pretty good for the root datum associated to
$G_{\overline{k}}$.
\end{definition}

This notion is related to the notions of good and very good primes for
root systems as follows:
\begin{lemma}\label{lem:G-vg}
Let $p$ be a prime number and $\R=(X,\Phi,Y,\Phi^\vee)$ as above.
Then the following hold:
\begin{itemize}
\item[(a)] $p$ is pretty good for $\R$ if and only if $p$ is good for $\Phi$ and the
groups $X/\ZZ\Phi$ and $Y/\ZZ\Phi^\vee$ have no $p$-torsion;
\item[(b)] $p$ is very good for $\Phi$ $\Rightarrow$ $p$ is pretty good for $\R$
$\Rightarrow$ $p$ is good for $\Phi$;
\item[(c)] if $\R$ is semisimple, then $p$ is pretty good for $\R$ if and only if
it is very good for $\Phi$;
\item[(d)] let $f:\R \rightarrow \tilde{\R}=(\tilde{X},\tilde{\Phi},\tilde{Y},
\tilde{\Phi}^\vee)$ be an isogeny of root data; if $\coker(f:X \rightarrow \tilde{X})$
has no $p$-torsion, then $p$ is pretty good for $\R$ if and only if it is pretty
good for $\tilde{\R}$;
\item[(e)] if $\R = \R_1 \oplus \R_2$, then $p$ is pretty good for $\R$ if and only
if it is so for both $\R_1$ and $\R_2$.
\end{itemize}
\end{lemma}

\begin{proof}
(a). For a subset $\Phi' \subseteq \Phi$ we consider the exact sequence
$$0 \rightarrow \ZZ\Phi / \ZZ\Phi' \rightarrow X/\ZZ\Phi' \rightarrow 
X/\ZZ\Phi \rightarrow 0.$$
If $p$ is pretty good, then the $p$-torsion of the middle group in this sequence
vanishes, so $p$ is good by Lemma \ref{lem:good} (a).
This proves the forward implication of (a).
For the reverse implication, the above exact sequence
and its obvious dual version
imply that the $p$-torsion
of $X/\ZZ\Phi'$ (resp.\ $Y/\ZZ\Phi'^\vee$) coincide with the $p$-torsion
of $\ZZ\Phi / \ZZ\Phi'$ (resp.\ $\ZZ\Phi^\vee / \ZZ\Phi'^\vee$). Since $p$ is good
for $\Phi$ and $\Phi^\vee$, the assertion follows again
from Lemma \ref{lem:good} (a).

(b). According to (a) it remains to show: if $p$ is very good for $\Phi$, then
$X/\ZZ\Phi$ and $Y/\ZZ\Phi^\vee$ have no $p$-torsion. 
For the first group, note that $(X \cap \QQ\Phi) / \ZZ\Phi$
has no $p$-torsion according to Lemma \ref{lem:good}
(b), since it is a subgroup of $\Lambda / \ZZ\Phi$. 
But then the same is true for $X/\ZZ\Phi$, since the quotient $X/(X\cap \QQ\Phi)$
has no torsion at all.
Since $p$ is also very good for $\Phi^\vee$, a
similar argument proves the assertion for $Y/\ZZ\Phi^\vee$.

(c). Suppose that $\R$ is a semisimple root datum.
According to (b) it remains to show that a pretty good prime $p$ for $\R$ 
does not divide the order of the fundamental group $\Lambda/\ZZ\Phi$.
In the semisimple setting,
the character lattice is a sublattice of the weight lattice, and 
the inclusion
$X \rightarrow \Lambda$ is dual to the inclusion $\ZZ\Phi^\vee \rightarrow Y$.
So both
$\Lambda / X$ and $X/\ZZ\Phi$ have no $p$-torsion by assumption,
and the assertion of (c) follows. 

(d). The condition on $f$ guarantees that the $p$-torsion of $X/\ZZ\Phi'$ coincides
with the $p$-torsion of $\tilde{X}/\ZZ f(\Phi')$, and similarly for $f^\vee$, whose
cokernel satisfies the same condition as the one of $f$.

(e). Suppose $\R_i = (X_i,\Phi_i,Y_i,\Phi_i^\vee)$ for $i=1,2$. 
Then any subset $\Phi' \subset \Phi = \Phi_1 \cup \Phi_2$ decomposes as
$\Phi' = \Phi'_1 \cup \Phi'_2$ with $\Phi'_i = \Phi' \cap \Phi_i$,
and $X / \ZZ\Phi' \cong X_1/ \ZZ\Phi'_1 \oplus X_2 /\ZZ\Phi'_2$. A similar decomposition
holds for $Y / \ZZ\Phi'^\vee$, which finishes the proof. 
\end{proof}

\begin{example}
The notion of a pretty good prime differs from the notions of good and very good primes.
For instance, every prime is pretty good for $\GL_2$, whereas $p=2$ is not very good for $\GL_2$.
On the other hand, all primes are good for $\SL_2$ and $\PSL_2$, whereas $p=2$ is not pretty good for $\SL_2$
or $\PSL_2$.
\end{example}

\section{Smooth centralizers}\label{sec:smoothcent}

Let $k$ be a field and $\overline{k}$ an algebraic closure of $k$. Note that
all constructions (preimages, centralizers, kernels, etc.) will be taken in the
scheme-theoretic sense of Section \ref{sec:prelim}.

\subsection{Relation to separability}\label{sec:sep} 
Suppose for the moment that $G$ is an algebraic group over the algebraically
closed field $\overline{k}$ of positive characteristic with Lie algebra $\g$.
This is the situation
considered in \cite{BMRT:2010}. 
There, a smooth closed subgroup $H\subseteq G$ is called \emph{separable in $G$}
if $\dim C_{G(\overline{k})}(H(\overline{k})) = \dim_{\overline{k}}
\g^{H(\overline{k})}$, where we consider the centralizer to be taken in the 
category of linear algebraic groups.
Similarly, a Lie subalgebra $\h \subseteq \g$ is called
\emph{separable in $\g$} if
$\dim \Cent_{G(\overline{k})}(\h) = \dim_{\overline{k}}
\cent_{\g}(\h)$.
Again we consider here the centralizer as a linear algebraic group, whereas
$\cent_{\g}({\h})$ is the Lie algebra centralizer of $\h$ in $\g$.

The next lemma explains the assertion of
\cite[Rem.\ 3.5(vi)]{BMRT:2010}, that relates the notion of separability to
the smoothness of certain centralizers of subgroup schemes. We use the notation
and equivalence of categories of Section \ref{subsec:smoothness}.
\begin{lemma} \label{lem:sep=smooth} Let $G$ be an algebraic group over
$\overline{k}$ with Lie algebra $\g$. Let $H$ be a smooth subgroup scheme 
of $G$ and let $\h \subseteq \g$ be a subalgebra.  
Then the following hold:
\begin{itemize} 
\item[(i)] The
centralizer $C_G(H)$ is smooth if and only if $H(\overline{k})$ is separable in
$G(\overline{k})$.  
\item[(ii)] Let $\h'$ be the $p$-envelope of $\h$ in $\g$,
i.e.\ the smallest $p$-subalgebra of $\g$ that contains $\h$. Let
$H'=\mathbb{G}(\h')$ be the affine group scheme associated to $\h'$ as in
Section \ref{subsec:tangent}. Then $C_G(H')$ is smooth if and only if $\h$ is
separable in $\g$.
\end{itemize} 
\end{lemma}

\begin{proof}
(i). Let $i: H \rightarrow G$ be the inclusion. The smoothness of $H$ implies
that for $g \in G(\overline{k})$, we have $\Int(g) = i: H \rightarrow G$ if and
only if $\Int(g)(\overline{k}) = i(\overline{k})$ (see \cite[Ex.\ 2.9 in
3.2]{Liu:2002}). Therefore, according to \eqref{eq:cent1}, we get the identity
\begin{align*} C_G(H)(\overline{k}) = C_{G(\overline{k})}(H(\overline{k})).
\end{align*} This implies that the dimensions of the scheme $C_G(H)$ and the
linear algebraic group $C_{G(\overline{k})}(H(\overline{k}))$ coincide (because
the defining ideal of the linear algebraic group now proves to be the radical
of the ideal defining $C_G(H)$). By \eqref{eq:lie_cent}, we also have 
that $\Lie(C_G(H)) =\g^H=\g^{H(\overline{k})}$ (for the last equation we note
that both subsets of $\g$ can be defined in terms of the same comodule map).
We conclude by Proposition \ref{prop:smoothness} that $C_G(H)$ is smooth if and
only if $\dim C_G(H) = \dim_{\overline{k}} \g^{H(\overline{k})}$, if and only
if $\dim C_{G(\overline{k})}(H(\overline{k})) = \dim_{\overline{k}}
\g^{H(\overline{k})}$. The last equality is equivalent to the separability of
$H(\overline{k})$ in $G(\overline{k})$.

(ii). As in part (i) we first get the identity $\Cent_G(\h)(\overline{k}) =
\Cent_{G(\overline{k})}(\h)$. We also have the identity $\Lie(\Cent_G(\h)) =
\cent_{\g}(\h)$ and can therefore conclude in the same manner as in part (i)
that $\h$ is separable in $\g$ if and only if $\Cent_G(\h)$ is smooth. Since
the actions of $G(\overline{k})$ and $\g$ on $\g$ are compatible with the
$p$-mapping, we get that $\h$ is separable in $\g$ if and only $\h'$ is
separable in $\g$ (in fact, $\cent_\g(\h)=\cent_\g(\h')$ and
$\Cent_{G(\overline{k})}(\h)=\Cent_{G(\overline{k})}(\h')$). It now suffices
to show that $\Cent_G(\h') = C_G(H')$. But for $g \in G(R)$ we have that
$\Ad(g) = di_R:\Lie(H'_R) \rightarrow \Lie(G_R)$ if and only if
$\Int(g)=i_R:H'_R \rightarrow G_R$ (see \eqref{eq:lie-adjoint}). So the desired
equality follows from the identities \eqref{eq:cent1} and \eqref{eq:cent2}.
\end{proof}

\subsection{Sufficient conditions for the smoothness of centralizers}
\label{sec:necsm}

A sufficient condition for the smoothness of fixed points is given by the
following theorem (see \cite[II, \S5, 2.8]{DG:1970}):

\begin{thm} \label{thm:lissite}
Let $G$ and $H$ be affine algebraic $k$-group schemes. Suppose that $G$ is
smooth and that $H$ acts via group automorphisms on $G$. If $H^1(H,\g)=0$ for
the corresponding $H$-module structure on $\g$, then $G^H$ is a smooth affine
algebraic group scheme. In particular, $G^H$ is smooth if $H$ is linearly
reductive.  
\end{thm}

Our goal is to prove the following theorem:

\begin{thm} \label{thm:main}
Let $G$ be a reductive algebraic group over $k$.
Suppose that the characteristic of $k$ is zero or pretty good for 
$G^\circ$. Suppose further that the component
group $\pi_0 G(\overline{k})$ has order invertible in $k$. Then for any closed
subgroup scheme $H\subseteq G$, the centralizer $C_G(H)$ is smooth. In
particular, all centralizers of closed subgroup schemes are smooth for a
connected reductive group $G$ in pretty good characteristic.  
\end{thm}

We defer the proof of the theorem to the end of this section. It relies
on a series of lemmas which generalize ideas of \cite{BMRT:2010} to the
scheme-theoretic setup. The main steps of the proof then follow closely the
proof of \cite[Thm.\ 1.2]{BMRT:2010}, plus a careful type $A$ analysis. 

\begin{remarks} \label{rem:cohom}

(a). Theorem \ref{thm:main} in particular answers a question raised in \cite[Rem.\ 3.5(vi)]{BMRT:2010},
attributed to Serre:
Even the centralizers of \emph{non-smooth} subgroup schemes of a connected reductive group $G$ are smooth
in very good characteristic.

(b). The cohomology criterion of Theorem \ref{thm:lissite} does not apply in
the situation of Theorem \ref{thm:main}. For example, take $k$ to be a field of
characteristic $3$. Let $H$ be the constant $k$-group scheme $\ZZ/3\ZZ$, realized
as a closed subgroup scheme of $\GL_2$ via 
\begin{align*} H(R) = \left \{
\begin{pmatrix} 1 & r\\ 0 & 1 \end{pmatrix} \Big\vert ~ r^3=r \right \}
\end{align*}
for each $k$-algebra $R$. Then $H^1(H,\gl_2)$ is one-dimensional as a
$k$-vector space (which can be checked by computing the cohomology of the
cyclic group of order $3$ with the corresponding action on $k^4$).
So the cohomology criterion does not apply for $\GL_2$, whereas all primes are
pretty good for
$\GL_2$.

(c). As pointed out to me by Steve Donkin, there is the following more generic
way to produce examples of non-vanishing cohomology on some $\gl_n$ as above.
Take any affine algebraic group scheme $G$ that is not linearly reductive. Then
there exist $G$-modules $M_1, M_2$ such that $\Ext^1_G(M_1,M_2)$ is not
trivial. Setting $V = M_1 \oplus M_2$, we find that $H^1(G, \gl(V)) = H^1(G,
V\otimes V^*) \supseteq H^1(G, M_2 \otimes M_1^*) \cong \Ext^1_G(k,M_2 \otimes
M_1^*) \cong \Ext^1_G(M_1,M_2)\neq 0$. If $V$ also happens to be a faithful
$G$-module (this can be arranged for example for any simple adjoint algebraic
group $G$ in positive characteristic), then we can identify $G$ with its image
in $\GL(V)$. This yields again an example where the theorem implies the
smoothness of the centralizer $C_{\GL(V)}(G)$, whereas the cohomology criterion
does not apply.

(d). The following example (\cite[Rem.\ 3.5]{BMRT:2010}) shows that the
assertion of the theorem may fail for non-connected reductive groups without
the hypothesis on the component group. Take $G = \Gm \rtimes \Aut(\Gm) = \Gm
\rtimes \ZZ/2\ZZ$. Then $C_G(\ZZ/2\ZZ) = \mu_2 \rtimes \ZZ/2\ZZ$ is smooth if
and only if $\cha(k)\neq 2$.

\end{remarks}

The following three lemmas are used in the proof of Theorem \ref{thm:main}. The
key idea is to transfer the desired property from $\GL_n$, where all
centralizers are smooth, according to the following first lemma.

\begin{lemma} \label{lem:caseGL}
Let $H$ be a closed subgroup scheme of $\GL_n$. Then the centralizer
$C_{\GL_n}(H)$ is smooth.
\end{lemma}

\begin{proof}
According to \eqref{eq:cent-base} and Proposition
\ref{prop:smoothness}, we can assume that $k=\overline{k}$ is algebraically
closed.
We have the identity $(\gl_n^H)_a = ((\gl_n)_a)^H$ (see \cite[II, \S2,
1.6]{DG:1970}). In particular, $((\gl_n)_a)^H$ can be represented by a reduced
$k$-algebra (see Section \ref{subsec:basic}). Let $R$ be a $k$-algebra and let
us identify $\gl_n \otimes  R= M_{n\times n}(R)$, where the right hand side
denotes $n\times n$-matrices with coefficients in $R$. We can write
\begin{align*} ((\gl_n)_a)^H (R) &= \{ x \in M_{n\times n}(R) \mid \Ad(h) (x_r)
= x_r \text{ for every }r:R \rightarrow R' \text{ and } h \in H(R')\} \\ &= \{
x \in M_{n\times n}(R) \mid  hx_rh^{-1} = x_r \text{ for every }r:R \rightarrow
R' \text{ and } h \in H(R')\},\\ C_{\GL_n}(H) (R) &= \{ g \in \GL_n(R) \mid
hg_r h^{-1} = g_r \text{ for every }r:R \rightarrow R' \text{ and } h \in
H(R')\} \\ &= \{ x \in ((\gl_n)_a)^H (R) \mid \det(x) \text{ is invertible} \}.
\end{align*} So $C_{\GL_n}(H)$ can be represented by the localization
$A_{\det}$, if $A$ represents $((\gl_n)_a)^H$. Since $A$ is reduced, the same
is true for $A_{\det}$. Hence $C_{\GL_n}(H)$ is smooth (Proposition
\ref{prop:smoothness}).  
\end{proof}

The next lemma shows how the smoothness of centralizers descends from one
reductive group to the other, provided that the inclusion induces a reductive
pair (see Section \ref{sec:prelim}).

\begin{lemma} \label{lem:red-pair}
Let $(G',G)$ be a reductive pair of reductive algebraic groups over $k$ and let
$H \subseteq G$ be a closed subgroup scheme. If $C_{G'}(H)$ is smooth, then so
is $C_G(H)$.
\end{lemma}

\begin{proof}
We proceed in $3$ steps. Let $\g' = \g \oplus \m$ be an $H$-module
decomposition of $\g'$.

1. It suffices to show that $c := \dim C_{G'}(H) - \dim C_G(H) \leq \dim_k
\m^H$. Indeed, by assumption (and using \eqref{eq:lie_cent}) 
we know that $\dim C_{G'}(H) = \dim_k \g'^H =
\dim_k \g^H + \dim_k \m^H$. So the desired inequality $\dim C_G(H) \geq \dim_k
\g^H$ (see Proposition \ref{prop:smoothness}) is equivalent to $c \leq \dim_k
\m^H$.

2. We show that $c \leq \dim_k T_{\overline e}((G'/G)^H)$. Consider the quotient $X
:= C_{G'}(H) / C_G(H)$. It is a smooth algebraic scheme of dimension $\dim X=
\dim_{\overline e} X = c$ (see \cite[III, \S3 2.7, 5.4 and 5.5(a)]{DG:1970}),
which can be formally obtained as the faisceau associated to the functor
\begin{align*} F_1: R \mapsto C_{G'}(H)(R) / C_{G}(H)(R), \end{align*} see
\cite[I, 5.6]{Jantzen:2003}. Now if we start with the functor $F_2: R \mapsto
G'(R) / G(R)$ we get as associated faisceau the smooth algebraic scheme $G'/G$
(which is in fact affine, due to the reductivity assumptions, see
\cite{Rich:1977}). Moreover, we have an inclusion of functors $F_2
\hookrightarrow G'/G$ (\cite[I, 5.6]{Jantzen:2003}). There is an obvious
action of $H$ on $F_2$ by conjugation, which extends uniquely to an action of
$H$ on $G'/G$ (see \cite[III, \S3, 1.3]{DG:1970}). We now have $F_1
\hookrightarrow F_2^H \hookrightarrow (G'/G)^H$, where the first inclusion
follows from the fact that $C_{G'}(H) \cap G = C_G(H)$. Therefore, we have a
monomorphism $X \hookrightarrow (G'/G)^H$ (see \cite[I, 5.4
(4)]{Jantzen:2003}). Note that $(G'/G)^H$ is closed in $G'/G$ (\emph{Ibid.}, I, 2.6 (10)),
hence also affine. We conclude that $c = \dim_{\overline e} X \leq
\dim_{\overline e} (G'/G)^H \leq \dim_k T_{\overline e}((G'/G)^H)$.  
 
3. We show that $T_{\overline e}(G'/G)^H \cong \m^H$ as $k$-vector spaces. Let us
first note that the natural map $\pi: G' \rightarrow G'/G$ is $H$-equivariant
with respect to the action defined in step 2 (because $\pi$ is the composite
$G' \rightarrow F_2 \rightarrow G'/G$). It is also flat with smooth fibre
$\pi^{-1}(\overline{e}) = G$ (see \cite[III, \S3, 2.5]{DG:1970}) and therefore
the induced map $d\pi_e: \g' \rightarrow T_{\overline e} (G'/G)$ is
$H$-equivariant and surjective, with kernel $\g$, according to Lemmas
\ref{lem:tangent-space} (c) and \ref{lem:surjective-diff}. So now $\m \cong
T_{\overline e} (G'/G)$ as $H$-modules, so that $\m^H \cong (T_{\overline e}
(G'/G))^H = T_{\overline e} ((G'/G)^H)$, where the last equality is Lemma
\ref{lem:tangent-space} (b).
\end{proof}

\begin{remark}
Let $H \subseteq G$ be a closed subgroup scheme of a smooth affine algebraic
$k$-group scheme and suppose that $H^1(H,\g)=0$. Then $C_G(H)$ is smooth,
according to Theorem \ref{thm:lissite}. We can also use the above two lemmas to
obtain a new proof for this fact: Pick a closed embedding $G \hookrightarrow
\GL_n$ for a suitable $n$ (see \cite[3.4]{Water:1979}). By Lemma
\ref{lem:caseGL}, $C_{\GL_n}(H)$ is smooth. If $H^1(H,\g)=0$, then $\dim_k
\gl_n^H = \dim_k \g^H + \dim_k(\gl_n/\g)^H$. Now the $3$ steps of the proof of
Lemma \ref{lem:red-pair} go through with $G'=\GL_n$ and $\m$ replaced by
$\gl_n/\g$ (we also need to extend our tangent space arguments to the possibly
non-affine variety $G'/G$).  
\end{remark}

The following final lemma states that the smoothness of a centralizer of a
subgroup scheme depends on the ambient group only up to central, separable
isogenies.

\begin{lemma} \label{lem:sep-isogeny}
Let $\pi: G \rightarrow G'$ be a central, separable isogeny of smooth affine algebraic $k$-group schemes,
and suppose that $H\subseteq G$ is a closed subgroup scheme. Then
$C_G(H)$ is smooth if and only if $C_{G'}(\pi H)$ is smooth. 
In particular, all centralizers in $G$ are smooth if and only if the same is true for $G'$.
\end{lemma}

\begin{proof}
We let $H$ act on $G'$ via $h.x = \pi(h)x\pi(h)^{-1}$ for $h \in H(R)$, $x \in G'(R)$.
In this way $\pi:G\rightarrow G'$ becomes an $H$-equivariant morphism.
Since $\ker \pi$ is smooth and finite,
the differential $d \pi:\g \rightarrow \g'$ is injective.
Moreover, according to our smoothness assumptions and Lemmas \ref{lem:tangent-space} (c) and
\ref{lem:surjective-diff}, 
it is also $H$-equivariant and surjective. Thus it induces an isomorphism $\g^H \cong \g'^H$.

We claim that $(G')^H = C_{G'}(\pi H)$.
The reverse inclusion follows from the definitions. Conversely,
suppose that $x \in (G')^H(R)$. Let $r: R \rightarrow R'$ be a $k$-algebra homomorphism and
let $h' \in \pi H(R')$.
We have to show that $h' x_r h'^{-1} = x_r$.
There is a faithfully flat extension $s:R' \rightarrow S$ and an element
$h \in H(S)$ such that $h'_s = \pi(h)$. By our choice of $x$ we get that
$x_{s\circ r} = h.x_{s\circ r} = h'_s x_{s\circ r} h_s'^{-1}$.
Now the injectivity of $s$ implies that $x_r = h'x_r h'^{-1}$, as required.

We now have established that $\Lie(C_G(H)) = \g^H \cong \g'^H = \Lie(C_{G'}(\pi H))$. To
prove the assertion, it therefore
remains to show that $\dim C_G(H) = \dim C_{G'}(\pi H)$.
Let $M = \pi^{-1} C_{G'}(\pi H) \subseteq G$. 
The map $\pi$ induces an exact sequence
\begin{equation*}
1 \rightarrow \ker(\pi) \rightarrow M \rightarrow C_{G'}(\pi H) \rightarrow 1.
\end{equation*}
It follows that $\dim M = \dim C_{G'}(\pi H)$. 
By construction, we have $C_G(H) \subseteq \pi^{-1} (G')^H = M$. 
To show that $\dim C_G(H) = \dim C_{G'}(\pi H)$, it therefore suffices to check
that $M^\circ \subseteq C_G(H)^\circ$.

Let $m \in M^\circ(R)$ and $h \in H(R')$ for some morphism of $k$-algebras $r:R \rightarrow R'$.
We need to show that $hm_r h^{-1}m_r^{-1} = e_{R'}$. 
Consider the natural map of $R'$-functors
\begin{gather*}
\varphi : M_{R'}^\circ \rightarrow (\ker \pi)_{R'} \\
M^\circ(R'') \ni x \mapsto h_s x h_s^{-1} x^{-1} \text{ for } s:R' \rightarrow R''. 
\end{gather*}
We first note that $\varphi$ is well-defined. Indeed, if $x \in M^\circ(R'')$, then $\pi(x)$ commutes with
all elements of $\pi H(R'')$. In particular, it commutes with $\pi(h_s)$. But then 
$h_s x h_s^{-1} x^{-1}$ has to lie in $\ker \pi(R'')$. In fact, $\varphi$ is even a morphism of affine
group schemes over $R'$, because we can compute that
\begin{equation*}
\varphi(xx') = h_s xx' h_s^{-1} x'^{-1} x^{-1} =
h_s x h_s^{-1} h_s x' h_s^{-1} x'^{-1} x^{-1} =
h_s x h_s^{-1} \varphi(x') x^{-1} = \varphi(x) \varphi(x'),
\end{equation*}
where for the last equation we use that $\ker(\pi) \subseteq Z(G)$.
Taking identity components, we get an induced homomorphism 
$\varphi^\circ: (M_{R'}^\circ)^\circ \rightarrow ((\ker \pi)_{R'})^\circ$.
Now, according to \eqref{eq:comp-base},
$(M_{R'}^\circ)^\circ = ((M^\circ)^\circ)_{R'} = (M^\circ)_{R'}$ and 
$((\ker \pi)_{R'})^\circ = ((\ker \pi)^\circ)_{R'} = \{e \}_{R'} = \{e_{R'} \}$, where
we use that $\ker \pi$ is smooth and finite.
So in fact our map $\varphi$ factors over $M_{R'}^\circ \rightarrow \{e_{R'} \}$. This implies that
$\varphi(m)=e_{R'}$, as required.

For the statement about all centralizers we just note that $\pi \pi^{-1} H' = H'$ for all closed subgroup schemes
$H' \subseteq G'$.
\end{proof}
We are now in a position to prove Theorem \ref{thm:main}. 

\begin{proof}[Proof of Theorem \ref{thm:main}] 
According to \eqref{eq:cent-base} and Proposition
\ref{prop:smoothness}, we can assume that $k=\overline{k}$ is algebraically
closed. Let $H$ be any closed subgroup scheme of $G$. Using the following idea
of \cite[2.8.1]{MT:2009}, we can first reduce to the case that $G$ is
connected: Let $H$ act via conjugation on $G$. Since $G^\circ$ is normal in $G$,
this induces an action of $H$ on $G^\circ$ and $(G^\circ)^H$ is a normal
subgroup of $C_G(H)$. This allows us to consider the exact sequence of affine
group schemes \begin{align*} 1 \rightarrow (G^\circ)^H \rightarrow C_G(H)
\rightarrow C_G(H)/(G^\circ)^H \rightarrow 1.  \end{align*} It now remains to
show that the left and right terms of this sequence are smooth (see Section
\ref{subsec:morphisms}). For the right hand side, note that the closed
embedding $C_G(H) \rightarrow G$ together with the fact that $C_G(H) \cap G^\circ =
(G^\circ)^H$ induces a closed embedding $C_G(H) / (G^\circ)^H \rightarrow
G/G^\circ$. Since $G/G^\circ$ is an \'etale group scheme (smooth and finite),
the same is true for $C_G(H) / (G^\circ)^H$ (see \cite[6.2]{Water:1979}). For
the left hand side, we claim that $(G^\circ)^H = (C_{G^\circ}(H\cap
G^\circ))^{\Gamma}$, where $\Gamma = H/(H\cap G^\circ)$. Indeed, $H$ acts via
conjugation on $C_{G^\circ}(H\cap G^\circ)$ with fixed points $(G^\circ)^H$.
But clearly, $H\cap G^\circ$ acts trivially, so the action factors through
$\Gamma$ and we get the asserted equality. Now, as a closed subgroup scheme of
$\pi_0G$, $\Gamma$ is again a smooth finite group scheme with order invertible
in $k$, hence linearly reductive. We can use Theorem \ref{thm:lissite} to
conclude that the fixed points are again a smooth group scheme, provided that
$C_{G^\circ}(H\cap G^\circ)$ is smooth. This allows us to assume from the
outset that $G$ is connected and, by hypothesis, that the characteristic of $k$
is zero or pretty good for $G$. 

If the characteristic is
zero, all closed subgroup schemes are smooth. So let us assume that $p=\cha(k)$
is a pretty good prime for $G$.

Let $G_1,\dots,G_t$ be the simple components of the derived group of $G$.
We may assume that $p$ is not very good for 
the root systems of
$G_1, \dots, G_n$, whereas it is for
$G_{n+1},\dots,G_{t}$. 
Since $p$ is good for $G$ and hence for all components of the derived group,
this implies that all groups $G_1, \dots, G_n$ are of type $A$.
Let $T$ be the radical of $G$ (which is a torus). Then multiplication in $G$
gives a central isogeny $T \times G_1 \times \dots \times G_t \rightarrow
G$. Denote by $\tilde{G}$ the image of $T \times G_1 \times \dots \times G_n$ in
$G$. Let $\pi: \tilde{G} \times G_{n+1} \times \dots \times G_t \rightarrow G$
be the induced central isogeny.
Due to our characteristic assumption, the Lie algebras $\Lie(G_i)$, $i>n$, are
simple (see \cite[Cor.\ 2.7]{Hoge:1982}). Using this, one can show that
$\Lie(\ker \pi) = \ker (d\pi)=0$, i.e.\ that  $\pi$ is separable. So, by Lemma
\ref{lem:sep-isogeny}, we may assume that $G=\tilde{G} \times G_{n+1} \times \dots \times
G_t$. 

Now let $H \subseteq G$  be a closed subgroup scheme. Then we find that 
$$C_G(H)
= C_{\tilde{G}}(\tilde{H}) \times C_{G_{n+1}}(H_{n+1}) \times \dots 
\times C_{G_t}(H_t),$$ 
where $H_i = p_i H$
for the projection $p_i:G \rightarrow G_i$,
and similarly for $\tilde{H}$.

We claim that the centralizers $C_{G_i}(H_i)$, $i>n$, are smooth.
In fact, since every central isogeny between
simple algebraic groups with the same root system in very good characteristic is separable, we
need to prove the assertion only for one algebraic simple group with a given
root system in very good characteristic.
According to \cite{Rich:1967}, for a given root system and an algebraically
closed field in very good characteristic, we can find a simple algebraic group
$G_i'$ of this type that fits into a reductive pair $(\GL_m,G_i')$ for a suitable
$m$. We may conclude that the $C_{G_i}(H_i)$, $i>n$, are smooth, by employing 
Lemmas \ref{lem:caseGL}, \ref{lem:red-pair} and \ref{lem:sep-isogeny}.  

It remains to show that centralizers in $\tilde{G}$ are smooth. By Lemma 
\ref{lem:G-vg} (d) and (e), $p$ is pretty good for $\tilde{G}$. By construction of $\tilde{G}$,
and by replacing $G$ by $\tilde{G}$, we may assume that $G$ is a connected reductive
group with root system $\Phi = A_{m_1} \times \dots \times A_{m_n}$, 
where $p$ divides all integers
$m_1 +1,\dots,m_n+1$. We also assume that the radical of $G$ has at least dimension $n$
(else we replace $G$ by $G \times S$, for some torus $S$, and exploit the fact that
$C_{G\times S}(H \times 1) = C_G(H) \times S$).

Let us consider now the reductive group $G_\CC$ over $\CC$ defined by the root datum
$(X,\Phi, Y, \Phi^\vee)$ of $G$. 
We identify $G_\CC$ and all other (smooth) complex algebraic groups below with
their groups of complex points. 
The prime $p$ now is just a fixed prime that is pretty good for the root datum
of $G_\CC$;
all we need is that it 
does not divide the order of the finite groups $Z(G_\CC)/\Rad(G_\CC)$ and 
$Z(G_\CC^\vee)/\Rad(G_\CC^\vee)$,
which follows from the fact that the abelian groups $X/\ZZ \Phi$ and $Y/\ZZ \Phi^\vee$ have
no $p$-torsion.

Let $r$ denote the dimension of the radical of $G_\CC$. By assumption, $r \geq n$. 
To ease notation we set
$\SL = \SL_{m_1 +1} \times \dots \times \SL_{m_n +1}$,
$\PSL = \PSL_{m_1 +1} \times \dots \times \PSL_{m_n +1}$ and
$\GL = \GL_{m_1 +1} \times \dots \times \GL_{m_n +1}$,
considered as complex algebraic groups. 
We now fix some data that determines $G_\CC$:
it is determined by its subgroups $\Der(G_\CC)$, $\Rad(G_\CC)$ and the central 
isogeny $\Der(G_\CC) \times \Rad(G_\CC) \rightarrow G_\CC$ given by multiplication.
In our setup this means that
there are normal, central, finite algebraic subgroups $M \subseteq N \subseteq 
\SL$ of 
$\SL$ and there is an injective homomorphism
$$\varphi: N/M \rightarrow \Gm^r$$ such that $G_\CC$ is (isomorphic to) the quotient of
$$N/M \rightarrow (\SL/M) \times \Gm^r, \quad
aM \mapsto (aM, \varphi(aM)^{-1}).$$
Note that $M$ is the kernel of the natural isogeny $\SL \rightarrow \Der(G_\CC)$, so
that $\SL/M \cong \Der(G_\CC)$, and $N/M \cong \Der(G_\CC) \cap \Rad(G_\CC)$.

Let $Z = Z(\SL)$. We claim that $p$ does not divide the orders of $M$ and $Z/N$. Indeed,
the order of $M = \ker(\SL \rightarrow \SL/M)$ is the same as the order of 
$\ker((\SL/M)^\vee 
\rightarrow \PSL)$, i.e.\ of $Z(\Der(G_\CC)^\vee)
= Z(G_\CC^\vee)/\Rad(G_\CC^\vee)$. On the other hand, 
\begin{align*}
Z/N &\cong (Z/M)/(N/M) \cong
Z(\Der(G_\CC))/(\Rad(G_\CC) \cap \Der(G_\CC)) \\
&\cong (Z(G_\CC)\cap\Der(G_\CC))/(\Rad(G_\CC)\cap\Der(G_\CC)) 
\cong Z(G_\CC)/\Rad(G_\CC).
\end{align*}
The assertion on $p$ follows.

Let $N_p$ denote the (unique) $p$-Sylow subgroup of $N$, and similarly $Z_p \subseteq Z$.
The last paragraph implies that the composite
$i: N_p \rightarrow N \rightarrow N/M$ is injective.
It also implies that $N_p=Z_p$.
Since $Z \cong \mu_{m_1+1} \times \dots \times \mu_{m_n+1}$, this gives
$N_p \cong \mu_{p^{s_1}} \times \dots \times \mu_{p^{s_n}}$, where $m_i +1=p^{s_i}m_i'$ with
$m_i'$ and $p$ coprime. In particular, the inclusion $N_p \rightarrow Z$ corresponds to
the standard inclusion
$\mu_{p^{s_1}} \times \dots \times \mu_{p^{s_n}} 
\rightarrow\mu_{m_1+1} \times \dots \times \mu_{m_n+1}$.
We embed $Z$ into $\Gm^n$ via the componentwise inclusion of $\mu_{m_1+1} \times
\dots \times \mu_{m_n +1}$ into $\Gm^n$.

The injective homomorphism $\varphi\circ i:N_p \rightarrow \Gm^r$ yields a 
$\ZZ$-linear surjection 
of character groups
$\ZZ^r \rightarrow \ZZ/p^{s_1} \times \dots \times \ZZ/p^{s_n}$. 
Choose a matrix $A \in \ZZ^{n \times r}$    
such that the above surjection is induced by
$$\ZZ^r \xrightarrow{A} \ZZ^n \rightarrow 
\ZZ/p^{s_1} \times \dots \times \ZZ/p^{s_n}$$ in the obvious way.
Since this concatenation is surjective, and since all $s_i$ are positive by assumption,
the reduction of $A$ mod $p$ is a surjective linear map $\Fp^r \rightarrow \Fp^n$ of 
$\Fp$-vector spaces. In particular, the elementary divisors $\alpha_1,\dots,
\alpha_n$ of $A$ over $\ZZ$ must all be prime to $p$.
Up to change of basis on both sides over $\ZZ$, $A$ decomposes as a linear map as
$\ZZ^r \rightarrow \ZZ^n \rightarrow \ZZ^n$, where the first map is the projection
onto the first $n$ components, and the second map sends $(x_1,\dots,x_n)$ to
$(\alpha_1 x_1,\dots,\alpha_n x_n)$.
Switching back to the category of multiplicative algebraic groups, we construct
the following commutative diagram:
\begin{equation} \label{eq:diag1}
\begin{tikzpicture}
[bij/.style={above,sloped,inner sep=0.5pt}]
\matrix (m) [matrix of math nodes, row sep=1.5em, column sep=3em,
text height=1.5ex, text depth=0.25ex]
	  {Z & N_p \\ 
      \Gm^n & & \Gm^r \\
      \Gm^n & \Gm^n & \Gm^r \\};
  \path[left hook->,font=\scriptsize]
  (m-1-2) edge (m-1-1) edge (m-2-1);
  \path[right hook->,font=\scriptsize] 
  (m-1-2) edge node[auto] {$\varphi \circ i$} (m-2-3)
  (m-1-1) edge (m-2-1);
  \path[->,font=\scriptsize]
  (m-2-1) edge node[auto] {$\varphi'$} (m-2-3)
    edge node[bij] {$\sim$} (m-3-1)
    edge[dotted] node[above,inner sep=3pt] {$\psi'$} (m-3-3)
  (m-3-3) edge node[bij]{$\sim$} node[right] {$\psi$} (m-2-3);
  \path[->>,font=\scriptsize] (m-3-1) edge node[below] {$\pi$} (m-3-2);
  \path[right hook->,font=\scriptsize](m-3-2) edge node[below] {$j$} (m-3-3);
\end{tikzpicture}
\end{equation}
Here, $\varphi'$ is the map induced by $A$, the isomorphisms correspond to the changes
of bases, $j$ is the inclusion into the first $n$ coordinates of $\Gm^r$, and $\pi$ is
a surjection with finite kernel of order prime to $p$ (this is due to the analysis 
of the elementary divisors above). In particular, the order of
$\ker(\psi') \cong \ker(\pi)$ is prime to $p$.

We finally consider the following diagram with exact rows, defining complex connected
reductive groups $H_1$ and $H_2$:
\begin{equation} \label{eq:diag2}
\begin{tikzpicture}
[equal/.style={above,sloped,inner sep=0.5pt}]
\matrix (m) [matrix of math nodes, row sep=1em, column sep=1em,
text height=1.5ex, text depth=0.25ex]
	  {N/M & (\SL/M) \times \Gm^r & G_\CC &1\\ 
    N_p & \SL \times \Gm^r & H_1 &1\\
    Z & \SL \times \Gm^r & H_2 &1\\
    Z & \SL \times \Gm^r & \GL \times \Gm^{r-n}&1  
    \\};
  \path[->,font=\scriptsize]
  (m-1-1) edge node[auto] {$i_1$} (m-1-2) 
  (m-1-2) edge (m-1-3) 
  (m-1-3) edge (m-1-4)
  (m-2-1) edge node[auto] {$i_2$} (m-2-2)
    edge node[auto] {$i$} (m-1-1)
  (m-2-2) edge (m-2-3) 
    edge node[auto] {$f$} (m-1-2)
    edge node[left] {$g$} (m-3-2)
  (m-2-3) edge (m-2-4)
  (m-3-1) edge node[auto] {$i_3$} (m-3-2)
  (m-3-2) edge (m-3-3) 
  (m-3-3) edge (m-3-4)
  (m-4-1) edge node[auto] {$i_4$} (m-4-2)
    edge node[equal] {$=$} (m-3-1)
  (m-4-2) edge (m-4-3) 
    edge node[auto] {$h$} (m-3-2)
  (m-4-3) edge (m-4-4);
  \path[right hook->] (m-2-1) edge (m-3-1);
\end{tikzpicture}
\end{equation}
Here, $i_1(aM) = (aM,\varphi(aM)^{-1})$, 
$i_2(a) = (a, \varphi'(a)^{-1})$,
$i_3(a) = (a, \varphi'(a)^{-1})$,
$i_4(a) = (a, j(a)^{-1})$; and
$f(x,t) = (xM,t)$,
$g(x,t) = (x,t)$,
$h(x,t) = (x,\varphi'(t_1,\dots,t_n)\psi(1,\dots,1,t_{n+1},\dots,t_{r}))$.
Per construction, all these maps are homomorphisms. The commutativity of \eqref{eq:diag1}
yields the commutativity of diagram \eqref{eq:diag2}, and moreover the surjectivity of
$h$. Since $f$ and $g$ are also surjective, these maps induce surjective 
homomorphisms $\overline{f}:H_1 \rightarrow G_\CC$,
$\overline{g}:H_1 \rightarrow H_2$ and
$\overline{h}:\GL \times \Gm^{r-n} \rightarrow H_2$.
For $(x,t) \in \SL \times \Gm^r$ let us write $\overline{(x,t)}$ for the image 
of $(x,t)$ in $H_1$, and also (with slight abuse of notation) for its image
in $\GL \times \Gm^{r-n}$.
Using the definitions, we can verify that the maps
\begin{align*}
\ker(\overline{f}) &\rightarrow N/N_p, \overline{(x,t)} \mapsto xN_p, \\
\ker(\overline{g}) &\rightarrow Z/N_p, \overline{(x,t)} \mapsto xN_p, \\
\ker(\overline{h}) &\rightarrow \ker(\psi'), \overline{(x,t)} \mapsto x(t_1,\dots,t_n) 
\end{align*}
are indeed well-defined inclusions of algebraic groups.
In particular, $\overline{f}, \overline{g}$ and $\overline{h}$ are \emph{isogenies with
kernels of order prime to $p$}.

To conclude, we express the diagram
of connected reductive groups and central isogenies
$$G_\CC \longleftarrow H_1 \longrightarrow H_2 \longleftarrow \GL
\times \Gm^{r-n}$$
in root data terms. This gives rise to connected reductive groups over $k$
and \emph{separable isogenies}
$$G \longleftarrow H_1' \longrightarrow H_2' \longleftarrow \GL \times \Gm^{r-n}.$$ 
We finish the proof by referring again to
Lemmas \ref{lem:caseGL} and \ref{lem:sep-isogeny}.  
\end{proof}

\section{Non-smooth centralizers}
\label{sec:nonsmooth}

In this section, let $k$ be an \emph{algebraically closed} field of
characteristic $p > 0$ and let $G$ be a connected reductive algebraic group
over $k$. Let $T$ be a maximal torus of $G$, with corresponding root datum
$(X,\Phi,Y,\Phi^\vee)$.

Our method to construct examples of closed subgroup schemes of $G$ with
non-smooth centralizers rests on the following easy observation.

\begin{lemma} \label{lem:torsion}
Suppose $H$ is a closed subgroup scheme of $G$ containing $T$. Then
$C=C_G(H)=Z(H)$ is diagonalizable. In particular, it is smooth if and only if
the $p$-torsion of its character group $X(C)$ vanishes.  \end{lemma}

\begin{proof}
We have $C=C_G(H) \subseteq C_G(T) = T \subseteq H$, so $C$ is diagonalizable
and coincides with the center of $H$. (For the  smoothness criterion see
\cite[IV, \S1, 1.2]{DG:1970}.) \end{proof}

\begin{example} \label{ex:counterexamples} 
(a). Suppose $p$ is a bad prime for $G$. Then, Lemmas \ref{lem:good} and
\ref{lem:torsion} suggest how to find a subgroup $H$ such that $C_G(H)$ is not
smooth: 

According to Lemma \ref{lem:good}, we can pick a subset $\Phi' \subseteq \Phi$
such that the $p$-torsion of $\ZZ \Phi / \ZZ \Phi'$ does not vanish. We may
assume that $\Phi'$ is closed and symmetric (else we replace it by $\ZZ \Phi'
\cap \Phi$). Now the connected reductive group $H$ generated by $T$ together
with the root groups $U_\alpha$, $\alpha \in \Phi'$ has center with character
group $X(T)/\ZZ \Phi'$ (see \cite[II, 1.6 and 1.7]{Jantzen:2003}), which
contains $\ZZ \Phi / \ZZ \Phi'$. So $C_G(H)$ is non-smooth, by Lemma
\ref{lem:torsion}.

Since we did not give a proof of Lemma \ref{lem:good}, it may be convenient to now
explicitly describe how to obtain a closed symmetric subset $\Phi'$ as above.
By definition of a bad prime, we can find an irreducible component $\Phi_1$ of
$\Phi$ such that $p$ divides a coefficient of the highest root $\beta$ of
$\Phi_1$ relative to some base $\Delta_1 = \{\alpha_1,\dots,\alpha_r\}$ of
$\Phi_1$. Suppose that the coefficient $m_1$ of $\alpha_1$ in $\beta$ is
divisible by $p$. Let $\Delta$ be a base of $\Phi$ containing $\Delta_1$. Let
$\Phi'$ be the  subroot system with base $\Delta'=(\Delta \setminus
\{\alpha_1\}) \cup \{-\beta\}$. Now the $p$-torsion of $\ZZ \Phi / \ZZ \Phi'$
equals  $\ZZ / p^r \ZZ$, where $p^r$ is the $p$-part of the integer $m_1$. Note
that in the language of the Borel-de Siebenthal algorithm (see e.g.\
\cite[2.1]{Good:2007}), we have obtained $\Phi'$ (and hence $H$), by ``crossing
out'' the node $\alpha_1$ in the extended Dynkin diagram corresponding to
$\Delta_1$, i.e.\ by crossing out a single node with a coefficient that is
divisible by $p$.  

(b). Suppose that $G$ has root system $\Phi=A_{m_1}\times \dots \times A_{m_n}$,
and that $p$ is not pretty good for $G$. Again, we want to find a subgroup $H$ such that
$C_G(H)$ is not smooth.

If the center $Z(G)$ is not smooth, we are done. Due to 
Lemma \ref{lem:G-vg} (a) 
(and because all primes are good for $\Phi$)
this means that we can assume that $Y/\ZZ\Phi^\vee$ has
$p$-torsion. Let $s = s_1 \cdots s_n$ be the element of the Weyl group of $G$ obtained
by multiplying Coxeter elements $s_i$ of the Weyl groups of type $A_{m_i}$.
We take here $s_i = s_{i1} \cdots s_{im_i}$, where the $s_{ij}$ are 
the simple reflections for some chosen root base 
$\Delta_i = \{ \alpha_{i1},\dots,\alpha_{i m_i} \}$
of $A_{m_i}$, which is ordered as in \cite[Planche I]{Bour:1968}.
Then $\Delta = \Delta_1 \cup \dots \cup \Delta_n$ is a root base of $\Phi$,
and hence a basis of $\ZZ\Phi$.
Let $\Delta^\vee$ be the corresponding dual base of $\Phi^\vee$.
For $\lambda \in X$ an easy induction on $n$ (and for $n=1$ on $m_1$) proves the equation
\begin{align*}
s(\lambda) = \lambda - \sum_{i=1}^n \sum_{j=1}^{m_i}
\sum_{k=j}^{m_i} \langle \lambda, \alpha_{ik}^\vee \rangle \alpha_{ij}.
\end{align*}
Consider the map $s-1 : X \rightarrow \ZZ \Phi$, 
$\lambda \mapsto s(\lambda)-\lambda$. We are going to compute the
elementary divisors of $(s-1)X$ in $\ZZ \Phi$.
Let $\mathfrak{b} = \{\lambda_1, \dots, \lambda_r\}$ be a basis of $X$ and 
$\mathfrak{b}^\vee = \{\lambda_1^\vee, \dots, \lambda_r^\vee \}$ a dual basis of $Y$ with respect
to $\langle,\rangle$.
Then the matrix $M$ of $(s-1)$ with respect to $\mathfrak{b}$
and $\Delta$ is given by the entries
\begin{align*}
M_{(i,j),l} = 
-\sum_{k=j}^{m_i} \langle \lambda_l, \alpha_{ik}^\vee \rangle,
\end{align*}
where $(i,j)$ varies in $i=1,\dots,n$, $j=1,\dots,m_i$, and $l=1,\dots,r$.
We have to compute the elementary divisors of $M$. We may replace
each $(i,j)$-th row by the $(i,j+1)$-th row minus the $(i,j)$-th row,
for $i=1,\dots,n$ and $j < m_i$, and obtain a matrix with entries
$\langle \lambda_l, \alpha_{ij}^\vee \rangle$.
But this is precisely the transpose of the matrix describing the map
$\ZZ \Phi^\vee \hookrightarrow Y$ with respect to the bases 
$\Delta^\vee$ and $\mathfrak{b}^\vee$.
We conclude that the elementary divisors of $(s-1)X$ in $\ZZ \Phi$ 
coincide with the elementary divisors of $\ZZ \Phi^\vee$ in $Y$,
so that the group $\ZZ\Phi/(s-1)X$ has $p$-torsion. Hence also the group
$X/(s-1)X$ has $p$-torsion.
This is the character group of $T^\sigma$, where 
$\sigma = s^{-1}$. Hence $T^\sigma$ is a non-smooth group scheme,
by the smoothness criterion referred to in the proof of Lemma \ref{lem:torsion}.
For a chosen representative $n_\sigma \in N_{G}(T)$, we may write this
group as $C_{G}(H)$, where $H$ is the subgroup generated by
$T$ and $n_\sigma$.
Our search has come to an end.
\end{example}

Combined with Theorem \ref{thm:main}, these examples allow us to prove the theorem
from the introduction:

\begin{proof}[Proof of Theorem \ref{thm:examples}]
The forward implication is clear by Theorem \ref{thm:main}. 
Now suppose all centralizers in $G$ are smooth.
According to Example \ref{ex:counterexamples}
(a), the characteristic needs to be good. Under this assumption we showed in the proof of 
Theorem \ref{thm:main}
that there is a separable isogeny $\tilde{G} \times H \rightarrow G$, where $H$ is defined
in very good characteristic, and $\tilde{G}$ is as in Example \ref{ex:counterexamples} (b).
If the characteristic was not pretty good for $\tilde{G}$, Example \ref{ex:counterexamples} (b)
would provide a subgroup scheme $\tilde{H}$ such that $C_{\tilde{G} \times H}(\tilde{H} \times 1)
= C_{\tilde{G}}(\tilde{H}) \times H$ would be non-smooth.
By Lemma \ref{lem:sep-isogeny}, this would give rise to a non-smooth centralizer in $G$,
contradicting our assumption. 
So the characteristic must be pretty good for $\tilde{G}$, and hence
for $\tilde{G} \times H$, by Lemma
 \ref{lem:G-vg} (e). According to part (d) of the same Lemma, it is also pretty good for $G$. 
\end{proof}

\begin{remark}
Theorem \ref{thm:examples} also shows that the existence of a reductive
pair $(\GL_n,G)$ (for some $n$) puts restrictions on the characteristic of $k$.
Indeed, suppose that the characteristic is not pretty good for $G$. Then the existence of
non-smooth centralizers implies (according to Lemmas \ref{lem:caseGL} and
\ref{lem:red-pair}) that we \emph{never} can find an inclusion $G \rightarrow
\GL_n$ that gives a reductive pair $(\GL_n,G)$. 
\end{remark}

\begin{remark}
The groups $H$ in Example \ref{ex:counterexamples} (a) and (b) have in common that
they are smooth, reductive (not necessarily connected), and contain a maximal torus of $G$. Using Lemma 
\ref{lem:sep=smooth} and Lemma \ref{lem:sep-isogeny}, the proof of Theorem
\ref{thm:examples} can be strengthened to yield the equivalence of the following three statements
about a connected reductive group $G$:
\begin{itemize}
\item[(a)] $\cha(k)$ is zero or pretty good for $G$, 
\item[(b)] all centralizers of closed subgroups schemes of $G$ are smooth,
\item[(c)] all (smooth) reductive subgroups of $G$ that contain a maximal torus are separable in $G$.
\end{itemize}

\end{remark}

\section{Standard reductive groups}
\label{sec:standard}

We finally want to relate our notion of pretty good primes to existing notions
of so called \emph{standard} reductive groups. Let $k$ be an algebraically closed field.
Consider the following conditions on a connected reductive group $G$ over $k$:

\begin{itemize}
\item $G$ is \emph{standard}, i.e.\  the derived group of $G$ is simply
connected, $\cha(k)$ is good for $G$, and $\Lie(G)$ carries a non-degenerate bilinear form
that is invariant under the $G$-action (these are the ``standard hypotheses'' of Jantzen,
see \cite[2.9]{Jantzen:2004}).
\item $G$ is \emph{$T$-standard}, i.e.\ $G$ is of the following form: it is separably isogenous 
to a Levi subgroup of $H$,
where $H$ is a smooth group scheme of the form $H =
H_1 \times T$ with a torus $T$ and a semisimple group $H_1$ for which the
characteristic of $k$ is very good (this notion is due to McNinch, see \cite{Mc:2008};
they are also called ``strongly standard'' in 
\cite{Mc:2005}).
\item $G$ is \emph{$D$-standard}, i.e.\  it is separably isogenous to a group $C_H(D)^\circ$, where $H$ is
as in the case above, 
and where $D \subseteq H$ is a (not necessarily smooth) subgroup scheme of
multiplicative type 
(this notion is due to McNinch and
Testerman, see \cite{MT:2009}; for smooth $D$ they are called ``strongly standard'' in 
\cite{MT:2007}).
We note here that, in fact, a group of
the form $C_H(D)^\circ$ is always reductive: it is smooth due to Theorem
\ref{thm:lissite}, and $H/C_H(D)$ is affine (\cite[XII, 5.5]{SGA3}), which
implies reductivity (\cite{Rich:1977}).
\end{itemize}

We slightly enlarge the classes of connected reductive groups defined by these properties:
We say that $G$ is \emph{essentially standard} (resp.\ \emph{essentially $T$-standard}, 
\emph{essentially $D$-standard}),
if it may be obtained from a standard (resp.\ $T$-standard, $D$-standard) group $H$ after a finite
number of the following operations: 
\begin{itemize}
\item[(i)] the replacement of $H$ by a separably isogenous group;
\item[(ii)] the replacement of $H=H' \times S$ by $H'$, where $S$ is a torus.
\end{itemize}

\begin{remark}
Cancelling off a torus factor really has some effect. For example, consider the groups
$G=\GL_5$ and $G'=\GL_5/\mu_2$, where $\mu_2$ is embedded in the center in the obvious way. Then
$G$ and $G'$ are non-isomorphic connected reductive groups.
However, the groups $G \times \Gm$ and $G' \times \Gm$ are isomorphic.
This can be verified by a small root data calculation.
It implies that operation (ii) applied to $G \times \Gm$ may produce $G'$.
\end{remark}

\begin{thm} \label{thm:standard}
Let $G$ be a connected reductive group defined over $k$. Then the following are equivalent:
\begin{itemize}
\item[(a)] The characteristic of $k$ is zero or pretty good for $G$,
\item[(b)] $G$ is essentially standard,
\item[(c)] $G$ is essentially $T$-standard,
\item[(d)] $G$ is essentially $D$-standard.
\end{itemize}
\end{thm}

\begin{proof}
Any Levi subgroup of a connected reductive group may be written as the centralizer of a suitable torus.
So clearly, (c) implies (d). 

Suppose (d) is satisfied. We claim that (a) holds.
According to Theorem \ref{thm:examples} it suffices to show that all centralizers in $G$ are 
smooth. 
By Lemma \ref{lem:sep-isogeny} we may assume that 
$G = C_H(D)^\circ$, where $D
\subseteq H$ is a diagonalizable subgroup scheme, and where $H$ is a reductive
group in very good characteristic.
It now suffices to show, by Theorem \ref{thm:main} and Lemma \ref{lem:red-pair}, that
$(H,C_H(D)^\circ)$ is a reductive pair.
We can finish the argument by employing
the ideas of \cite[Prop.\ 3.7 (a)]{BMRT:2010}: As a $D$-module, the Lie
algebra $\h = \Lie(H)$ decomposes into a direct sum $\h = \bigoplus_\chi
\h_\chi$, where $\chi$ varies over the character group of $D$ (see
\cite[II, \S2, 2.5]{DG:1970}). Moreover, $\h_0 = \h^D = \Lie(C_H(D)) =
\Lie(C_H(D)^\circ)$, so that the reductive pair condition is satisfied.

Suppose that (a) holds. In the proof of Theorem \ref{thm:main} we have shown, that up to
separable isogenies and direct products with a torus, $G$ is a direct product of 
groups of the following form: a torus, some $\GL_n$, a simple group in very good
characteristic. Since $\GL_n$ is $T$-standard (see e.g.\ \cite[Rem.\ 3]{Mc:2005}), the whole
direct product is $T$-standard, and (c) holds.

Suppose $G$ is $T$-standard. Then $G$ is separably isogenous to a standard group, by
\cite[Prop.\ 2]{Mc:2005}. This means that (c) implies (b).

Suppose that $G$ is standard. Then the center $Z(G)$ is smooth (by the argument of 
\cite[3.4.2]{MT:2009}: $Z(G)=C_G(X)^\lambda$ for some $X \in \Lie(G)$ and some
cocharacter $\lambda$, and Jantzen's conditions guarantee that the centralizer $C_G(X)$ is
smooth). The same is true for $Z(G^\vee)$, for example since $Z(G^\vee)/\Rad(G^\vee)$ is
trivial due to the fact that $\Der(G)$ is simply connected. Since $p$ is also good,
it is pretty good, by Lemma \ref{lem:G-vg} (a). 
Since the class of groups in pretty good characteristic is closed under the operations 
(i) and (ii) above, any essentially standard group is now defined in pretty good characteristic.
We have shown that (b) implies (a), which finishes
  the proof.
\end{proof}

By Theorem \ref{thm:examples}, we deduce the following consequence.

\begin{corollary} \label{cor:univsm}
Let $G$ be a connected reductive group satisfying any of the three standardness conditions above.
Then all centralizers of closed subgroup schemes in $G$ are smooth.
\end{corollary}

\begin{remark}
George McNinch suggested to give an easier direct description of essentially standard groups, starting from
the easiest examples, and operations like (i) and (ii) above. We can do this as follows. Consider
the class of connected reductive groups with the following properties:
\begin{itemize}
\item it contains all simple groups defined in very good characteristic;
\item it contains $\GL_n$ for all $n$;
\item it contains all tori;
\item it is closed under taking products and operations (i), (ii) as above.
\end{itemize}
By Lemma \ref{lem:G-vg}, all groups in this class are defined in pretty good characteristic
(for $\GL_n$ we use part (a) of the Lemma and the fact that in this case $X(T)/\ZZ\Phi \cong \ZZ \cong Y(T)/\ZZ\Phi^\vee$). 
The converse follows from the proof of the implication $(a) \Rightarrow (c)$ in Theorem \ref{thm:standard} above.

By the characterization in terms of the smoothness of all centralizers, the class of groups defined in
pretty good characteristic is closed under taking subgroups such that we obtain a reductive pair. In
particular, it is closed under taking centralizers of diagonalizable subgroup schemes (cf.\ the proof
of $(d) \Rightarrow (a)$ in Theorem \ref{thm:standard}).
\end{remark}

The last remark gives a very easy recipe to prove results for groups defined in pretty good characteristic.
We illustrate this with the existence of \emph{Springer isomorphisms}.

\begin{corollary}
Let $G$ be a connected reductive group defined in pretty good characteristic. Then there is a $G$-equivariant
isomorphism of varieties $\U \rightarrow \N$, where $\U\subseteq G$ (resp. $\N\subseteq \g$) is the unipotent (resp.\
nilpotent) variety associated to $G$.
\end{corollary}

\begin{proof}
The result holds for $\GL_n$ and for tori.
In \cite[Prop.\ 9]{Mc:2005}, McNinch shows how to extend Springer's work to $T$-standard groups,
by first extending the result to semisimple groups defined in very good characteristic. 
He also shows that the existence of Springer isomorphisms
is compatible with operation (i). The compatibility with operation (ii) and with taking direct products is obvious.
\end{proof}

\bigskip
{\bf Acknowledgements}: This paper was prepared towards the author's PhD
qualification under the supervision of Gerhard R\"ohrle, with financial support
of the DFG-priority program SPP1388 ``Representation Theory''. I thank Gerhard
R\"ohrle for help in preparing this paper. I am also very grateful to Olivier
Brunat for helpful discussions, and to Steve Donkin for pointing out the
example in Remark \ref{rem:cohom} (c). I am indebted to George McNinch for comments and
questions
on an earlier version of this manuscript.
Finally, I thank the referee for numerous improvements of the exposition.


\begin{thebibliography}{amsalpha}
\bibitem[BMR05]{BMR:2005}
M.~Bate, B.~Martin, and G.~R{\"o}hrle.
\newblock A geometric approach to complete reducibility.
\newblock {\em Invent. Math.}, 161(1):177--218, 2005.

\bibitem[BMRT]{BMRT:2010}
M.~Bate, B.~Martin, G.~R{\"o}hrle, and R.~Tange.
\newblock Complete reducibility and separability.
\newblock {\em Trans. Amer. Math. Soc.}, 362(8):4283--4311, 2010.

\bibitem[Bo91]{Borel:1991}
A.~Borel.
\newblock {\em Linear algebraic groups}, volume 126 of {\em Graduate Texts in
  Mathematics}.
\newblock Springer-Verlag, New York, second edition, 1991.

\bibitem[Bou68]{Bour:1968}
N.~Bourbaki. 
\newblock \emph{\'{E}l\'ements de math\'ematique. {F}asc. {XXXIV}. {G}roupes
  et alg\`ebres de {L}ie. {C}hapitre {IV}: {G}roupes de {C}oxeter et syst\`emes
  de {T}its. {C}hapitre {V}: {G}roupes engendr\'es par des r\'eflexions.
  {C}hapitre {VI}: syst\`emes de racines}. 
\newblock Actualit\'es Scientifiques et
  Industrielles, No. 1337, Hermann, Paris, 1968.

\bibitem[DG70]{DG:1970}
M.~Demazure and P.~Gabriel.
\newblock {\em Groupes alg\'ebriques. {T}ome {I}: {G}\'eom\'etrie alg\'ebrique,
  g\'en\'eralit\'es, groupes commutatifs}.
\newblock Masson \& Cie, \'Editeur, Paris, 1970.
\newblock Avec un appendice {sur Corps de classes local} par Michiel
  Hazewinkel.

\bibitem[SGA3]{SGA3}
M.~Demazure and A.~Grothendieck~(dirig.).
\newblock {\em Sch\'emas en groupes, S\'eminaire de G\'eom\'etrie Alg\'ebrique
  du Bois Marie 1962/64 (SGA 3)}.
\newblock (Lecture Notes in Math., \textbf{151--153}). Springer-Verlag, Berlin,
  1970.

\bibitem[Go07]{Good:2007}
S.~M. Goodwin.
\newblock On generation of the root lattice by roots.
\newblock {\em Math. Proc. Cambridge Philos. Soc.}, 142(1):41--45, 2007.

\bibitem[Ho82]{Hoge:1982}
G.~M.~D. Hogeweij.
\newblock Almost-classical {L}ie algebras. {I}, {II}.
\newblock {\em Nederl. Akad. Wetensch. Indag. Math.}, 44(4):441--452, 453--460,
  1982.

\bibitem[Ja03]{Jantzen:2003}
J.~C. Jantzen.
\newblock {\em Representations of algebraic groups}, volume 107 of {\em
  Mathematical Surveys and Monographs}.
\newblock American Mathematical Society, Providence, RI, second edition, 2003.

\bibitem[Ja04]{Jantzen:2004}
J.~C. Jantzen.
\newblock Nilpotent orbits in representation theory.
\newblock In {\em Lie theory}, volume 228 of {\em Progr. Math.}, pages 1--211.
  Birkh\"auser Boston, Boston, MA, 2004.

\bibitem[KMRT]{KMRT:1998}
M-A. Knus, A.~Merkurjev, M.~Rost, and J-P. Tignol.
\newblock {\em The book of involutions}, volume~44 of {\em American
  Mathematical Society Colloquium Publications}.
\newblock American Mathematical Society, Providence, RI, 1998.
\newblock With a preface in French by J. Tits.

\bibitem[LT99]{LT:1999}
R.~Lawther and D.~M. Testerman.
\newblock {$A_1$} subgroups of exceptional algebraic groups.
\newblock {\em Mem. Amer. Math. Soc.}, 141(674):viii+131, 1999.

\bibitem[LS96]{LS:1996}
M.~W. Liebeck and G.~M. Seitz.
\newblock Reductive subgroups of exceptional algebraic groups.
\newblock {\em Mem. Amer. Math. Soc.}, 121(580):vi+111, 1996.

\bibitem[Liu02]{Liu:2002}
Q.~Liu.
\newblock {\em Algebraic geometry and arithmetic curves}, volume~6 of {\em
  Oxford Graduate Texts in Mathematics}.
\newblock Oxford University Press, Oxford, 2002.
\newblock Translated from the French by Reinie Ern{\'e}, Oxford Science
  Publications.

\bibitem[McN05]{Mc:2005}
G.~J. McNinch.
\newblock Optimal {${\rm SL}(2)$}-homomorphisms.
\newblock {\em Comment. Math. Helv.}, 80(2):391--426, 2005.

\bibitem[McN08]{Mc:2008}
G.~J. McNinch.
\newblock The centralizer of a nilpotent section.
\newblock {\em Nagoya Math. J.}, 190:129--181, 2008.

\bibitem[MT07]{MT:2007}
G.~J. McNinch and D.~M. Testerman.
\newblock Completely reducible {$\rm SL(2)$}-homomorphisms.
\newblock {\em Trans. Amer. Math. Soc.}, 359(9):4489--4510 (electronic), 2007.

\bibitem[MT09]{MT:2009}
G.~J. McNinch and D.~M. Testerman.
\newblock Nilpotent centralizers and {S}pringer isomorphisms.
\newblock {\em J. Pure Appl. Algebra}, 213(7):1346--1363, 2009.

\bibitem[Ri67]{Rich:1967}
R.~W. Richardson.
\newblock Conjugacy classes in {L}ie algebras and algebraic groups.
\newblock {\em Ann. of Math. (2)}, 86:1--15, 1967.

\bibitem[Ri77]{Rich:1977}
R.~W. Richardson.
\newblock Affine coset spaces of reductive algebraic groups.
\newblock {\em Bull. London Math. Soc.}, 9(1):38--41, 1977.

\bibitem[Sp98]{Springer:1998}
T.~A. Springer.
\newblock {\em Linear algebraic groups}, volume~9 of {\em Progress in
  Mathematics}.
\newblock Birkh\"auser Boston Inc., Boston, MA, second edition, 1998.

\bibitem[SpSt70]{SpSt:1970}
T.~A. Springer and R.~Steinberg.
\newblock Conjugacy classes.
\newblock In {\em Seminar on {A}lgebraic {G}roups and {R}elated {F}inite
  {G}roups ({T}he {I}nstitute for {A}dvanced {S}tudy, {P}rinceton, {N}.{J}.,
  1968/69)}, Lecture Notes in Mathematics, Vol. 131, pages 167--266. Springer,
  Berlin, 1970.

\bibitem[Wat79]{Water:1979}
W.~C. Waterhouse.
\newblock {\em Introduction to affine group schemes}, volume~66 of {\em
  Graduate Texts in Mathematics}.
\newblock Springer-Verlag, New York, 1979.

\end{thebibliography}
\end{document}